\documentclass{article}

\usepackage{graphicx}

\usepackage[fleqn]{amsmath}
\usepackage{amsmath}
\usepackage{amsthm}
\usepackage{amssymb}
\usepackage{amsbsy}
\usepackage{amsfonts}
\usepackage{mathrsfs}
\usepackage[all]{xy}
\usepackage{amstext}
\usepackage{amscd}
\usepackage[dvips]{epsfig}
\usepackage{psfrag}
\usepackage{enumerate}
\usepackage{flafter}

\allowdisplaybreaks

\theoremstyle{plain}
\newtheorem{thm}{Theorem}[section]
\newtheorem{prop}[thm]{Proposition}

\newtheorem{lem}[thm]{Lemma}
\theoremstyle{definition}
\newtheorem{exa}[thm]{Example}

\newtheorem{rem}[thm]{Remark}

\def\Im{\mathop{\mathrm{Im}}\nolimits}
\def\Ker{\mathop{\mathrm{Ker}}\nolimits}
\def\Coker{\mathop{\mathrm{Coker}}\nolimits}

\def\F{\mathop{\mathbb{F}}\nolimits}

\newcommand{\lra}{\longrightarrow}
\newcommand{\ra}{\rightarrow}

\newcommand{\Z}{{\Bbb Z}}

\newcommand{\D}{{\cal D}_{g}}

\newcommand{\DD}{{{\cal D}_{g}}}

\newcommand{\M}{{\cal M}}

\newcommand{\X}{\widetilde{X} }

\newcommand{\OX}{{\rm O}(X) }

\begin{document}
\large
\begin{center}{\bf\Large Central extensions of groups and adjoint groups of quandles }
\end{center}
\vskip 1.5pc
\begin{center}{\Large Takefumi Nosaka}\end{center}
\begin{abstract} \baselineskip=13pt \noindent
This paper develops an approach for describing centrally extended groups, as determining the adjoint
groups associated with quandles. Furthermore, we explicitly describe such groups of some quandles. As a
corollary, we determine some second quandle homology groups.
\end{abstract}

\baselineskip=15pt 
\section{Introduction}  
A quandle is a set with a binary operation $\lhd :X^2 \rightarrow X$ such that
\begin{enumerate}
\item The identity $ a \lhd a = a $ holds for any $a \in X $.
\item The map $ (\bullet \lhd a ): \ X \rightarrow X$ defined by $x \mapsto x \lhd a $ is bijective for any $a \in X$.
\item The identity $(a\lhd b)\lhd c=(a\lhd c)\lhd (b\lhd c)$ holds for any $a,b,c \in X. $
\end{enumerate}
The axioms are partially motivated by knot theory and braidings.
For example, any group is a quandle by the operation $a \lhd b:= b^{-1}ab;$
see \S \ref{Sin2} for other examples.
Conversely, for a quandle $X$, we can define {\it the adjoint group}
as the following group presentation:
$$ \mathrm{As}(X)= \langle \ e_x \ \ (x \in X )\ \ | \ \ e_{x\lhd y}^{-1} \cdot e_y^{-1} \cdot e_x \cdot e_y \ \ (x,y \in X)\ \ \rangle.$$
It is known that the correspondence $X \mapsto \mathrm{As} (X)$ yields a functor from the category of quandles to that of groups with left-adjointness.

Let us briefly explain some significance to determine $\mathrm{As} (X)$.
First, when the adjoint functor $F : \mathcal{C} \rightleftharpoons \mathcal{C}' : G$ is given, 
the difference between $\mathrm{id}_{\mathcal{C} } $ and $G \circ F$ is important in some areas (see \cite{Ger,Nos3} for the quantum representations, or see \cite{AG,CJKLS,FRS1,Kab,Joy,Nos1} and references therein for knot-invariants and pointed Hopf algebras), with a relation to centrally extended groups.
Furthermore, as a result of Eisermann \cite{Eis2} (see also Theorem \ref{ehyouj2i3}),
the second quandle homology $H_2^Q(X) $ can be computed from concrete expressions of $\mathrm{As}(X) $.
Moreover, the main theorem in \cite{Nos2} shows that
certain universal knot-invariants from quandles turn out to be characterized by the third group homology $ H_3^{\mathrm{gr}}( \mathrm{As} (X)) $.
However, as we see the definition of $\mathrm{As}(X)$ or some explicit computation \cite{Cla2}, it has been considered to be hard to deal with $\mathrm{As} (X)$ concretely.

In this paper, we develop a method for formulating practically $\mathrm{As} (X)$ in a purely algebraic way.
This method is roughly summarized to `universal central extensions of groups modulo type-torsion' (see \S \ref{s3s3w}--\ref{SS3j3});
the main theorem \ref{cccor1} emphasizes importance of the concept of types.
Furthermore, Section \ref{Sin2} demonstrates practical applications of the method.
Actually, we succeed in determining $\mathrm{As} (X)$ and the associated homology $H_2^Q(X) $ of some quandles $X$ (up to torsion).
As a special case, Subsection \ref{Ex5} compares the theorem with Howlett's theorem \cite{How} concerning the Schur multipliers of Coxeter groups.
Furthermore, in Section \ref{SScon32}, we will see that the method is applicable to coverings in quandle theory.

\vskip 0.7pc
\noindent
{\bf Notation and convention.}
For a group $G$, we denote by $H_n^{\rm gr}(G)$
the usual group homology in trivial integral coefficients.
Moreover, a homomorphism $f:A \rightarrow B$ between abelian groups is said to be
{\it a $[1/N]$-isomorphism} and is denoted by $ f: A \cong_{[1/N]} B $,
if the localization of $f$ at $\ell$ is an isomorphism for any prime $\ell$ that does not divide $N$.
This paper does not need any basic knowledge in quandle theory, but assumes basic facts of group cohomology as in \cite[Sections I, II and VII]{Bro}.

\section{Preliminaries and the main theorem}\label{s3s3w}
This section aims to state Theorem \ref{cccor1}.
We start by reviewing properties of quandles.
A quandle $X$ is said to be {\it of type $t_X $}, if $t_X >0$ is the minimal $N$ such that $x= x \lhd^N y $ for any $x,y \in X$,
where we denote by $\bullet \lhd^N y$ the $N $-times on the right operation with $y$.
Note that, if $X$ is of finite order, it is of type $t_X $ for some $t_X \in \Z $.

Next, let us study the adjoint group $\mathrm{As} (X)$ in some details.
Define a right action $\mathrm{As}(X)$ on $X$ by $ x \cdot e_y:= x \lhd y$ for $x,y \in X$.
Notice the equality
\begin{equation}\label{hasz2} e_{x \cdot g} = g^{-1} e_x g \in \mathrm{As} (X) \ \ \ \ \ \ \ (x \in X , \ \ g \in \mathrm{As} (X) ), \end{equation}
by definitions.
The orbits of this action of $\mathrm{As}(X)$ on $X$ are called {\it connected components of $X$}, denoted by $\mathrm{O}(X).$
For $i \in \mathrm{O}(X)$, we let $X_i \subset X$ be the orbit with respect to $i$.
If the action is transitive (i.e., $\mathrm{O}(X)$ is a singleton), $X$ is said to be {\it connected}.
Furthermore, with respect to $i \in \OX $, define a homomorphism
\begin{equation}\label{epsilon} \displaystyle{ \varepsilon_i : \mathrm{As} (X) \lra \Z } \ \ \ \ \textrm{ by } \left\{ \begin{array}{ll}
\displaystyle{ \varepsilon_i ( e_x)=1 \in \Z }, &\ \mathrm{if} \ \ x \in X_i, \\
\displaystyle{ \varepsilon_i ( e_x)=0\in \Z }, &\ \mathrm{if} \ \ x \in X \setminus X_i.\\
\end{array} \right. \ \end{equation}
Note that the direct sum $\oplus_{i \in \OX} \varepsilon_i$ yields the abelianization $\mathrm{As}(X)_{\rm ab} \cong \Z^{\oplus \OX}$ by \eqref{hasz2}, which means that the group $\mathrm{As}(X)$ is of infinite order.
Furthermore, if $O(X)$ is single, we often omit writing the index $i$.

In addition, we briefly review
{\it the inner automorphism group}, $\mathrm{Inn}(X)$, of a quandle $X$.
Regard the action of $\mathrm{As} (X)$ as a group homomorphism $\psi_X$ from $\mathrm{As}(X)$ to the symmetric group $\mathrm{Bij}(X,X)$.
The group $\mathrm{Inn}(X) $ is defined as the image $\mathrm{Im}(\psi_X ) \subset \mathrm{Bij}(X,X) $.
Hence, we have a group extension
\begin{equation}\label{AI} 0 \lra \Ker (\psi_X) \lra \mathrm{As} (X) \xrightarrow{\ \psi_X \ } \mathrm{Inn}(X) \lra 0 \ \ \ \ \ (\mathrm{exact}).
\end{equation}
By the equality \eqref{hasz2}, this kernel $ \Ker (\psi_X)$ is contained in the center.
Therefore, it is natural to focus on their second group homology; we show a theorem on $ H_2^{\rm gr} (\mathrm{As}(X)) $ as a useful estimate:

\begin{thm}\label{cccor1}
For any connected quandle $X$ of type $t_X $ (possibly, $X$ could be of infinite order), the second group homology $ H_2^{\rm gr} (\mathrm{As}(X)) $ is annihilated by $t_X $.
Furthermore, the abelian kernel $\Ker (\psi_X)$ in \eqref{AI} is $[1/t_X]$-isomorphic to $\Z \oplus H_2^{\rm gr} (\mathrm{Inn} (X)) $.
\end{thm}
The proof will appear in \S \ref{kokoroni}. 
In conclusion,
metaphorically speaking, $\mathrm{As}(X)$ turns out to be the `universal central extension' of $ \mathrm{Inn}(X)$ up to $t_X $-torsion;
hence,
this theorem emphasizes importance of the concept of types; so as to investigate $\mathrm{As}(X)$, we shall study $\mathrm{Inn}(X)$ and $ H_2^{\rm gr}(\mathrm{Inn}(X))$. 

\section{Methods on inner automorphism groups. }\label{SS3j3}
Following the previous theorem to study the group $\mathrm{As}(X)$,
we shall develop a method for describing the inner automorphism group $\mathrm{Inn}(X)$:
\begin{thm}\label{keykantan1g}
Let a group $G$ act on a quandle $X$.
Let a map $\kappa : X \rightarrow G$ satisfy the followings:
\begin{enumerate}
\item The identity $ x \lhd y = x \cdot \kappa (y) \in X$ holds for any $x,y \in X$.
\item The image $\kappa (X) \subset G$ generates the group $G$, and the action $X \curvearrowleft G$ is effective.
\end{enumerate}
Then, there is an isomorphism $\mathrm{Inn}(X) \cong G$, and the action $X \curvearrowleft G $ agrees with
the natural action of $ \mathrm{Inn}(X)$.
\end{thm}
\begin{proof}
Identify the action $X \curvearrowleft G$ with a group homomorphism $f: G \rightarrow \mathrm{Bij}(X,X) $.
It follows from the first assumption that $f (\kappa(X)) \subset \mathrm{Inn}(X)$ and $f (\kappa(X)) $ generates $\mathrm{Inn}(X)$;
thus, $f$ gives rise to an epimorphism $ F: \langle \kappa(X)\rangle \rightarrow \mathrm{Inn}(X)$,
where $\langle \kappa(X)\rangle $ is the subgroup of $G$ generated by $\kappa(X)$.
Then, the secibd assumption assures that $ \langle \kappa(X)\rangle =G$, and the second implies the bijectivity of $F$,
i.e., $ \mathrm{Inn}(X) \cong G$. Moreover, the agreement of the two actions follows by construction.
\end{proof}
This theorem is applicable to many quandles, in practice. Actually, as seen in Section \ref{Sin2}, we can determine $ \mathrm{Inn}(X)$ of many quandles $X$.
However, we here explain that this theorem is inspired by the Cartan embeddings in symmetric space theory as follows:
\begin{exa}\label{lemspsoinn2}
Let $X$ be a symmetric space in differential geometry. 
Consider the group $\mathrm{Inn}(X) \subset \mathrm{Diff}(X) $ generated by the symmetries $\bullet \lhd y $ with compact-open topology.
As is well known, $\mathrm{Inn}(X)$ has a Lie group structure, and
the map $ X \rightarrow \mathrm{Inn}(X)$ that sends $y$ to $s_y$ is commonly called the Cartan embedding.
As seen in textbooks on symmetric spaces, Theorem \ref{keykantan1g} had been used to determine $\mathrm{Inn}(X)$ concretely.
\end{exa}
Furthermore, we suggest another computation when $\mathrm{Inn}(X) $ is perfect. 
\begin{prop}\label{vanithm2}
Let $X$ be a quandle, and $O(X)$ be the set of orbits of the action $X \curvearrowleft \mathrm{As}(X)$. 
Set the epimorphism $\varepsilon_i: \mathrm{As}(X)\rightarrow \Z $ associated with $i \in O(X)$ defined in \eqref{epsilon}.
If the group $\mathrm{Inn}(X)$ is perfect, i.e., $H_1^{\mathrm{gr}}(\mathrm{Inn}(X)) =0$,
then we have an isomorphism
\begin{equation}\label{asin} \mathrm{As} (X) \cong \Ker (\oplus_{i \in O(X)}\varepsilon_i )\times\Z^{\oplus O(X)}, \end{equation}
and this $ \Ker (\oplus_{i \in O(X)}\varepsilon_i ) $ is a central extension of $\mathrm{Inn}(X) $ and is perfect.
In particular, if $X$ is connected and the group homology $H_2^{\mathrm{gr}} ( \mathrm{Inn}(X) )$ vanishes, then $\mathrm{As} (X) \cong \mathrm{Inn}(X) \times \Z$.
\end{prop}
\begin{proof}We will show the isomorphism \eqref{asin}. 
By the assumption $H_1^{\mathrm{gr} }( \mathrm{Inn}(X)) =0$, 
the composite $ \Ker ( \psi_X ) \hookrightarrow \mathrm{As}(X) \stackrel{\mathrm{proj.}}{\lra} H_1^{\mathrm{gr} }(\mathrm{As}(X)) =\Z^{\oplus O(X)} $ obtained from \eqref{AI}
is surjective.
Since $ \Z^{\oplus O(X)}$ is free, we can choose a section $\mathfrak{s}: \Z^{\oplus O(X)} \rightarrow \Ker ( \psi_X )$ of the composite. Hence, by the equality \eqref{hasz2} and the inclusion $ \Ker ( \psi_X ) \subset \mathrm{As}(X) $, the semi-direct product $\mathrm{As}(X) \cong \Ker (\oplus_{i \in O(X)}\varepsilon_i ) \rtimes \Z^{\oplus O(X)} $ is trivial, leading to \eqref{asin} as desired.
Furthermore the kernel $ \Ker (\oplus_{i \in O(X)}\varepsilon_i ) $ is a central extension of $\mathrm{Inn}(X) $ by construction, and is perfect by the Kunneth theorem and $\mathrm{As}(X)_{\rm ab} \cong \Z^{\oplus O(X)}$.
Hence we complete the proof.
\end{proof}
\begin{rem}\label{df032}In general, the kernel $ \Ker (\oplus_{i \in O(X)}\varepsilon_i ) $ is not always the universal central extension of the perfect group $\mathrm{Inn}(X)$;
see \cite[Theorem 4]{Nos3} with $g=3$ as a counterexample such that the extension $ \Ker (\oplus_{i \in O(X)}\varepsilon_i ) \rightarrow \mathrm{Inn}(X)$ is not universal.
\end{rem}

Finally, we conclude this section by giving two lemmas on $\mathrm{As}(X)$, which are used later.
\begin{lem}\label{daiji} Let $X$ be a connected quandle of type $t < \infty $.
Then, for any $x,y \in X$, we have the identity $ (e_x)^{t}=(e_y)^{t}$ in the center of $\mathrm{As}(X)$.
\end{lem}
\begin{proof}For every $b \in X$, note the equalities $(e_x)^{-t} e_b e_x^{t} = e_{(\cdots (b \lhd x ) \cdots )\lhd x} = e_b $ in $\mathrm{As}(X)$.
Namely $ (e_x)^{t}$ lies in the center. Furthermore the connectivity admits
$ g \in \mathrm{As}(X)$ such that $ x \cdot g = y$.
Hence, it follows from \eqref{hasz2} that $(e_x)^{t} = g^{-1} (e_x)^{t} g = (e_{ x \cdot g} ) ^{t}= (e_y)^{t} $ as desired.
\end{proof}
\begin{lem}\label{lem11}
Let $X $ be a connected quandle of finite order.
Then its type $t_X $ is a divisor of $ |\mathrm{Inn}(X)| /|X|$.
\end{lem}
\begin{proof}
For $x,y \in X$, we define $m_{x,y}$ as the minimal $n$
satisfying $x\lhd^n y =x$. Note that $ (\bullet \lhd^{m_{x,y}} y) $ lies in the stabilizer $\mathrm{Stab}(x)$.
Since $ |\mathrm{Stab}(x)| = |\mathrm{Inn}(X)| /|X|$ by connectivity,
any $m_{x,y}$ divides $|\mathrm{Inn}(X)| /|X| $; hence so does the type $t_X $.
\end{proof}
Furthermore, in some cases,
we can calculate some torsion parts of their group homology:
\begin{lem}\label{dfgk232}Let $X$ be a connected quandle of type $t_X $.
If $H_2^{\rm gr}(\mathrm{Inn}(X))$ is annihilated by $t_X < \infty $, then there is a $[1/t_X]$-isomorphism
$H^{\rm gr}_3(\mathrm{Adj}(X)) \cong H^{\rm gr}_3 (\mathrm{Inn}(X))$. 
\end{lem}
\begin{proof}
Consider the Lyndon-Hochschild spectral sequence of \eqref{AI}.
It is sufficient for the proof to show that the differential
$$d_2 : E^2_{3,0} = H_3({\rm Inn}(X);H_0({\rm Ker}(\psi_X)) ) \lra E^2_{1,1} = H_1^{\rm gr}({\rm Inn}(X);H_1^{\rm gr}({\rm Ker}( \psi_X ) ))$$
is trivial modulo $t_X$.
Indeed, the inflation-restriction sequence of \eqref{AI}
$$\mathrm{Ker}(\psi_X) \lra H_1^{\rm gr}({\rm As}(X);\Z) \lra H_1^{\rm gr}({\rm Inn}(X);\Z) \lra 0 $$
and Lemma \ref{daiji} imply that $H_1^{\rm gr}({\rm Inn}(X);\Z)$ is annihilated by $t_X$.
Furthermore, we obtain the $[1/t_X]$-isomorphism $\mathrm{Ker}(\psi_X) \cong_{[1/t_X]} \Z $ from Theorem \ref{cccor1}. Hence, the image of $d_2$ is trivial as required.
\end{proof}

\section{Six examples of $\mathrm{As}(X) $ and second quandle homology }\label{Sin2}
Based on the previous results on $\mathrm{As}(X)$, this section calculates $\mathrm{Inn}(X)$ and $\mathrm{As} (X)$ for six kinds of connected quandles $X$:
Alexander, symplectic, spherical, Dehn, Coxeter and core quandles.
These quandles are dealt with in six subsections in turn.

Furthermore, to determine the second quandle homology $ H_2^Q (X)$ in trivial $\Z$-coefficients (see \S \ref{kokoroni} for the definition),
we will employ the following computation of Eisermann: 
\begin{thm}[{\cite[Theorem 1.15]{Eis2}}]\label{ehyouj2i3}
Let $X $ be a connected quandle.
Fix an element $x_0 \in X$. Let $ \mathrm{Stab}(x_0) \subset \mathrm{As}(X) $ be the stabilizer of $x_0$, and $\varepsilon : \mathrm{As} (X)\rightarrow \Z$ be the abelianization mentioned in \eqref{epsilon}.
Then, 
$H_2^Q (X )$ is isomorphic to the abelianization of $\mathrm{Stab}(x_0) \cap \Ker (\varepsilon) $.
\end{thm}

\subsection{Alexander quandles}\label{Ex1}
We start by discussing the class of Alexander quandles.
Every $\Z[T^{\pm 1}] $-module $X $ has a quandle structure with the operation
$ x\lhd y=y +T( x- y) $ for $x,y\in X$, and is called {\it the Alexander quandle}.
This operation $\bullet \lhd y$ is roughly a $T$-multiple centered at $y$.
The type is the minimal $N$ such that $T^N=\mathrm{id}_X$ since $ x\lhd^n y= y+ T^n (x -y) $.
Furthermore, it can be easily verified that an Alexander quandle $X$ is connected if and only if $(1-T)X=X$.

Let us review the concrete presentation of $\mathrm{As}(X)$, which is due to Clauwens \cite{Cla2}.
When $X$ is connected,
set up the homomorphism $\mu_X : X \otimes X \rightarrow X \otimes X$ defined by $ \mu_X( x \otimes y)= x\otimes y - T y \otimes x.$
Further, he defined a group operation on $\Z \times X \times \mathrm{Coker}( \mu_X)$ by setting
$$ (n, x,\alpha) \cdot (m, y,\beta)= (n+m, T^m x +y, \ \alpha + \beta + [T^m x \otimes y]),$$
and constructed a group isomorphism $\mathrm{As}(X) \rightarrow \Z \times X \times \mathrm{Coker}( \mu_X)$,
which takes $e_x$ to $(1,x,0)$.
As a result, we see that the kernel of $\psi_X : \mathrm{As} (X) \rightarrow \mathrm{Inn}(X)$ equals $t_X \Z \times \Coker (\mu_X)$.

Thanks to his presentation of $\mathrm{As}(X)$, we can easily show a result of Clauwens that determines the homology $H_2^Q(X)$ of a connected Alexander quandle $X$.
To be precise,
\begin{prop}[{Clauwens \cite{Cla2}}]\label{H2dfg31sha}
Let $X$ be a connected Alexander quandle.
The homology $H_2^Q(X) $ is isomorphic to the quotient module $\Coker (\mu_X)= X\otimes_{\Z} X /(x \otimes y -Ty \otimes x)_{x,y \in X}$.
\end{prop}
\begin{proof}By definition we can see that the
$\Ker (\varepsilon) \cap \mathrm{Stab }(0)$ is the cokernel $\mathrm{Coker}( \mu_X)$.
\end{proof}

\subsection{Symplectic quandles}\label{Ex2}
Let $K$ be a commutative field, and let $\Sigma_g$ be the closed surface of genus $g$.
Consider the multiplicative group $ K^{\times }$, and the quotient $ K^{\times }/(K^{\times })^2$ modulo 2.
For $[r] \in K^{\times }/(K^{\times })^2$, we fix a representative $r \in K^{\times }$,
and consider the copy of $H^1(\Sigma_g; K ) \setminus \{ 0 \}= K^{2g} \setminus \{ 0 \}$,
denoted by $X_r $.
Let $X$ be the union $\cup_{r \in K^{\times } /(K^{\times })^2} X_r$ (here, we should notice that $X=X_r$ if $K$ is an algebraically closed field.). 
Using the standard symplectic 2-form $\langle, \rangle: H^1(\Sigma_g; K) \times H^1(\Sigma_g; K) \rightarrow K$,
the set $X $ is made into a quandle by the operation $ x \lhd y :=r  \langle x,y \rangle y +x \in X$ for $x \in X$ and $y \in X_r$, and is called {\it a symplectic quandle (over $K $)}.
The operation $ \bullet \lhd y: X\rightarrow X$ is commonly called
{\it the transvection} of $y$.
Note that the type of the quandle $X$ is the characteristic of $K$ since $ x\lhd^N y= N r  \langle x,y \rangle y +x $.

We will determine $ \mathrm{Inn}(X )$ and $ \mathrm{As} (X )$ associated with the symplectic quandle $X$ over $K$.
\begin{lem}\label{lemspsoinn}
$ \mathrm{Inn}(X )$ is isomorphic to the symplectic group $Sp(2g;K)$.
\end{lem}
\begin{proof}
Recall from the Cartan-Dieudon\'{e} theorem, that the classical group $Sp(2g;K)$ is generated by transvections $(\bullet \lhd y)$.

We will show the desired isomorphism.
For any $y \in X$, the map $ ( \bullet \lhd y): X \rightarrow X $ is a restriction of a linear map $K^{2g} \rightarrow K^{2g}$.
It thus yields a map $\kappa : X \rightarrow GL(2g;K ) $, which factors through $ Sp(2g; K )$ and satisfies the conditions in Theorem \ref{keykantan1g}.
Indeed, the condition (II) follows from the classical theorem and
the effectivity of the standard action $K^{2g} \curvearrowleft Sp(2g; K ) $.
Therefore $ \mathrm{Inn}( X ) \cong Sp(2g; K ) $ as desired.
\end{proof}

\begin{prop}\label{vanithm2444}
Take a field $K$ of positive characteristic $p$ and with $|K| > 10$.
Assume the connectivity, that is, every $ x \in K$ admits a square $\sqrt{x}$ in $K.$
Let $X = K^2 \setminus \{ 0,0 \}$ be the symplectic quandle over $K$, and $\widetilde{Sp}(2g;K)$ be the
universal central extension of $Sp(2g;K)$.
Then $\mathrm{As} (X) \cong \Z \times \widetilde{Sp}(2g;K)$.
\end{prop}
\begin{proof}
Since $X$ is connected and $\mathrm{Inn}(X) \cong Sp(2g;K) $ by Lemma \ref{lemspsoinn}, Proposition \ref{vanithm2} implies $\mathrm{As} (X) \cong \Ker (\varepsilon )\times \Z$.
Further, it follows from Theorem \ref{cccor1} that $H_2^{\rm gr}(\mathrm{As} (X))$ is annihilated by $p$.
Hence, following the fact \cite{Sus} that $ H_2^{\rm gr}(Sp(2g;K) )$ has no $p$-torsion,
the kernel $\Ker (\varepsilon ) $ must be the
universal central extension of $Sp(2g;K)$, which completes the proof.
\end{proof}
\begin{rem} This proposition holds even if the characteristic of $K$ is zero and $X$ is not connected; see \cite{Nos2} for the proof.
Furthermore, the paper \cite{Nos2} also determines $ H_2^Q(X)$ in the case where $K$ is of infinite order,
\end{rem}

Accordingly, hereafter, we will focus on finite fields $K= \F_q$ with $q >10$:
\begin{prop}\label{dfg31speak} Let $X$ be
the symplectic quandle over $\F_q$. If $q>10$, then $\mathrm{As}(X) \cong \Z^{O(X)} \times Sp(2g;\mathbb{F}_q)$.
Furthermore, $H_3^{\mathrm{gr}}(\mathrm{As}(X)) \cong \Z/(q^2-1)$.
\end{prop}
\begin{proof}
Since $(\F_q)^{\times} $ is cyclic, 
we first should notice that, if $q$ is even $|O(X)| =1$, 
and that, if $q$ is odd, $|O(X)| =2$ or $1$ according to $q=4r+1$ or $q=4r+3$ for some $r \in \Z$.

Since $q>10$, the first and second homology groups of $\mathrm{Inn}( X) \cong Sp(2g;\mathbb{F}_q)$ are known to be zero (see \cite{FP,Fri}).
Thus, $\widetilde{Sp}(2g;\F_q) = Sp(2g;\F_q) $, leading to $\mathrm{As}(X) \cong \Z^{O(X)} \times Sp(2g;\mathbb{F}_q)$ as stated.
Furthermore, the latter part follows from the result $H_3^{\mathrm{gr}}(Sp(2g;\mathbb{F}_q)) \cong \Z/(q^2-1)$ in \cite{FP,Fri}.
\end{proof}
As a result, we will determine the second homology $H_2^Q(X) $.
\begin{prop}\label{dfg31spk} Let $q >10$, and $X$ be as above.
If $g \geq 2$, the homology $H_2^Q(X) $ vanishes.
If $g=1$, then $H_2^Q(X) \cong (\Z/ p )^{d|O(X)|} $, where $q =p^d$.
\end{prop}
\begin{proof}
Recall $\mathrm{As}(X) \cong \Z^{O(X)} \times Sp(2g;\mathbb{F}_q)$.
Considering the standard action $X \curvearrowleft Sp(2g; \F_q) $, denote by $G_X $ the
stabilizer of $(1,0, \dots, 0) \in ( \mathbb{F}_q)^{2g}$.
Since Theorem \ref{ehyouj2i3} immediately means $ H_2^Q(X) \cong H_1^{\rm gr}(G_X )$,
we will calculate $H_1^{\rm gr}(G_X)$ as follows.
First, for $g =1$, it can be verified that the stabilizer $G_X $ is exactly the product $(\Z/ p )^d $ as an abelian group;
hence $ H_2^Q(X) \cong ( \Z/ p )^{d|O(X)| } $ in the sequel.
Next, for $g \geq 2$, the vanishing $ H_2^Q(X) = H_1^{\rm gr}(G_X) =0$ immediately follows from Lemma \ref{d;apk} below.
\end{proof}

\begin{lem}\label{d;apk} Let $g \geq 2$ and $q >10$.
Let $G_X$ denote the stabilizer of the action $X \curvearrowleft Sp(2g;\F_q)$ mentioned above.
Then the homology groups $ H_1^{\rm gr}(G_X) $ and $ H_2^{\rm gr}(G_X )$ vanish.
\end{lem}
\begin{proof}
Since $q>10$, recall from \cite[II. \S 6.3]{FP} the order of $ Sp(2g; \F_q)$ as
$$ |Sp(2g; \F_q) | = q^{g^2 }(q^{2g}-1) (q^{2g-2}-1) \cdots (q^2-1) . $$
Since $|X|= q^{2g}-1$, the order of $G_X$ is equal to $ q^{g^2} \cdot |Sp(2g-2; \F_q) |$.
Thereby $H_1^{\rm gr}(G_X)$ and $H_2^{\rm gr}(G_X )$ are zero up to $p$-torsion,
because of the inclusion $ Sp(2g-2; \F_q) \subset G_X$ by definitions and the vanishing $H_1^{\rm gr} \oplus H_2^{\rm gr} (Sp(2g-2; \F_q)) \cong 0$ up to $p$ torsion.

Finally, we may focus on the $p$-torsion of $ H_1^{\rm gr} \oplus H_2^{\rm gr}(G_X) $.
Following the proof of \cite[Proposition 4.4]{Fri}, there is a certain subgroup
``$\Delta(Sp(2g;\F_q))$" of $G_X $ which contains a $p$-sylow group of $Sp(2g;\F_q)$ and
this $\Z/p$-homology vanishes. Hence, $ H_1^{\rm gr} \oplus H_2^{\rm gr}(G_X) =0$ as required.
\end{proof}

\subsection{Spherical quandles}\label{Ex3}
Let $K $ be a field of characteristic not equal to $2$, and fix $n \geq 2$ in this subsection.
Take the standard symmetric bilinear form $ \langle, \rangle : K ^{n+1} \otimes K^{n+1} \rightarrow K $.
Consider a set of the form
$$ S^{n}_K := \{\ x \in K^{n+1} \ | \ \langle x, x \rangle =1\ \}. $$
We define the operation $ x \lhd y $ to be $ 2 \langle x,y \rangle y -x \in S^{n}_K $.
The pair $(S^{n}_K, \ \lhd )$ is a quandle of type $2$, and is referred to as {\it a spherical quandle} (over $K $).
This operation $\bullet \lhd y$ can be interpreted as
a linear transformation which identically acts on $y$ and $-\mathrm{Id}$ on the the subspace orthogonal tof $y$.

Then, similar to the proof of Lemma \ref{lemspsoinn}, one can readily determine $\mathrm{Inn}(X)$ as follows:
\begin{lem}\label{lemspsoinn2444}
If $n$ is odd, then $ \mathrm{Inn}(S^{n}_K )$ is isomorphic to the orthogonal group $O(n+1;K)$. 
If $n$ is even, $ \mathrm{Inn}(S^{n}_K )$ is isomorphic to $SO(n+1;K)$. 
\end{lem}

Next, we will focus on second homology group and $H_3^{\rm gr}(\mathrm{As}(X))$ of spherical quandles over $\F_q $.
Here, the results are up to $2$-torsion, whereas the 2-torsion part is the future problem.
\begin{prop}\label{dfg31sha}
Let $X$ be a spherical quandle over $\F_q$. Let $q >10$.
For $n \geq 3$, the second homology $H_2^Q(X) $ is annihilated by $2$.
If $n=1$, then the homology $H_2^Q(X) $ is $[1/2]$-isomorphic to the cyclic group $\Z/ (q- \delta_q ) $,
where $ \delta_q = \pm 1$ is according to $q \equiv \pm 1 (\mathrm{mod \ } 4)$.
\end{prop}
\begin{proof}
Assume $n$ is odd. 
Under the standard action $X \curvearrowleft O(n+1; \F_q) $,
the stabilizer of $(1, 0, \dots, 0 ) \in X$ is $O(n ; \F_q)$.
By a similar discussion to the proof of Proposition \ref{dfg31spk},
$ H_2^Q(X) \cong H_1^{\rm gr} ( O(n ; \F_q) ) $ modulo 2-torsion.
For $n \geq 3$, the abelianization of $O(n ; \F_q)$ is $(\Z/2)^2$; see \cite[II. \S 3]{FP}; hence
the $H_2^Q(X) $ is annihilated by $2$ as required.
The same discussion in the even case $n$ works well, since the inclusion $SO(n) \rightarrow O(n)$ induces $H_*(SO(n ; \F_q)) \cong_{[1/2]} H_* (O(n ; \F_q))$ modulo 2-torsion.

Finally, when $n=1 $, the group $O(2 ; \F_q)$ is cyclic and of order $q- \delta_q$.
Hence $H_2^Q(X) \cong H_1^{\rm gr} ( O(2 ; \F_q) ) \cong_{[1/2]} \Z/(q- \delta_q )$.
\end{proof}

\begin{prop}\label{dfg31sh} Let $q >10$.
Then $H_3^{\rm gr}(\mathrm{As}(X)) \cong_{[1/2]} H_3^{\rm gr} ( O(n+1 ;\mathbb{F}_q))$ up to 2-torsion.
\end{prop}
\begin{proof} Since $q>10$, $H_1^{\rm gr} \oplus H_2^{\rm gr}(O(n+1 ;\mathbb{F}_q))$ is known to be annihilated by $2$; see \cite{Fri,FP}.
Hence, the conclusion readily results from Lemma \ref{dfgk232}.
\end{proof}

\subsection{Dehn quandle}\label{Ex4}
Changing the subject,
we now review Dehn quandle \cite{Y}.
Denote by $ \mathcal{M}_{g}$ the mapping class group of $\Sigma_{g} $, and consider the set, $\D$, defined by
\begin{equation}\label{defdg} \notag \D:= \{ \ \textrm{ isotopy classes of (unoriented) non-separating simple closed curves } \gamma \ {\rm in \ } \Sigma_g \ \}. \end{equation}
For $ \alpha, \ \beta \in \D $, we define $\alpha \lhd \beta \in \D$ by $ \tau_{\beta}(\alpha)$,
where $ \tau_{\beta} \in \M_{g} $ is the positive Dehn twist along $\beta$.
The pair ($\D, \lhd)$ is a quandle, and called {\it (non}-{\it separating) Dehn quandle}.
As is well-known, any two non-separating simple closed curves are conjugate by the product of some Dehn twists. Hence, the quandle $\D$ is connected, and is not of any type $t $.
The Dehn quandle $\D$ is applicable to study 4-dimensional Lefschetz fibrations (see, e.g., \cite{Y,Zab,Nos3}).
The natural inclusion $\kappa: \D \rightarrow \M_g$ implies $\mathrm{Inn}(\D) \cong \M_g$ by Theorem \ref{keykantan1g}.
Furthermore, if $g \geq 4$, there is an isomorphism $\mathrm{As} (\DD) \cong \Z \times \mathcal{T}_g$ shown by \cite{Ger},
where $\mathcal{T}_g$ is the universal central extension of $\M_g$ associated with $H_2^{\rm gr}(\M_g) \cong \Z$.

The result of this subsection is the following:
\begin{prop}\label{eahges} 
If $g \geq 5$, then $H_2^Q(\DD) \cong \Z/2$.
\end{prop}
\begin{proof}
We will use the facts that an epimorphism $G \rightarrow H$ between groups induces an epimorphism $G_{\rm ab} \rightarrow H_{\rm ab}$, and that $\M_{g,r}$ is perfect.

Fixing $\alpha \in \DD$, we begin by observing the stabilizer $\mathrm{Stab}(\alpha) \subset \mathrm{As} (\DD) $.
Note that the map $\DD \rightarrow \M_g $ sending $\beta $ to $\tau_{\beta}$
yields a group epimorphism $\pi: \mathrm{As} (\DD) \rightarrow \M_g $.
Furthermore, by Proposition \ref{vanithm2}, the restriction of $ \pi$ to $ \Ker (\varepsilon ) \cong \mathcal{T}_{g}$ coincides with the projection $ \mathcal{T}_{g} \rightarrow \M_g$.
In particular, we thus have $\pi(\mathrm{Stab}(\alpha) ) = \pi(\mathrm{Stab}(\alpha) \cap \Ker (\varepsilon ) ) \subset \M_g $.

We will construct a surjection $ H_2^Q(\DD) \rightarrow \Z/2$.
By the virtue of Theorem \ref{ehyouj2i3},
it is enough to construct a surjection from the previous $\pi(\mathrm{Stab}(\alpha) \cap \Ker (\varepsilon ) ) $ to $\Z/2$ for $g \geq 2$.
As is shown \cite[Proposition 7.4]{PR}, we have the following exact sequence:
\begin{equation}\label{kkkl}0 \lra \Z \lra \mathcal{M}_{g-1,2 } \stackrel{\xi}{\lra} \pi (\mathrm{Stab}(\alpha) ) \stackrel{\lambda}{\lra} \Z/2 \ \ \ \ \ \ \ \ \ \mathrm{(exact)} . \end{equation}

\noindent
Here $\xi$ is the homomorphism induced from the gluing $ (\Sigma_{g-1,2}, \partial (\Sigma_{g-1,2}) ) \ra(\Sigma_{g}, \alpha ) $, and $\lambda$ is defined by the transposition of the connected components of boundaries of $\Sigma_g \setminus \alpha$.
By considering a hyper-elliptic involution preserving the above $\alpha$, the map $\lambda$ is surjective.
Hence $ \pi(\mathrm{Stab}(\alpha) \cap \Ker (\varepsilon ) ) $ surjects onto $\Z/2$ as desired.

Finally, we will complete the proof.
By Theorem \ref{ehyouj2i3} again, recall that $\bigl( \mathrm{Stab}(\alpha) \cap \Ker (\varepsilon ) \bigr)_{\rm ab} \cong H_2^Q(\DD) $.
To compute this, put the inclusion $\iota :\pi(\mathrm{Stab}(\alpha) ) \rightarrow \mathcal{M}_{g}$.
By the Harer-Ivanov stability theorem (see \cite{Iva}), the composition $ \iota \circ \xi:\mathcal{M}_{g-1,2} \rightarrow \mathcal{M}_{g}$ induces
an epimorphism
\begin{equation}\label{kkdkl} (\iota \circ \xi)_*: H_2^{\mathrm{gr}} ( \mathcal{M}_{g-1,2};\Z ) \lra H_2^{\mathrm{gr}}(\mathcal{M}_{g};\Z)\ \ \ \ \ \ \mathrm{for} \ \ g\geq 5. \end{equation}
Since $H_2^{\mathrm{gr}} ( \mathcal{M}_{g-1,2};\Z ) \cong H_2^{\mathrm{gr}}(\mathcal{M}_{g};\Z) \cong \Z$ is known (see, e.g., \cite{FM}), the epimorphism \eqref{kkdkl} is isomorphic.
Let $ (\iota \circ \xi)^* (\mathcal{T}_g)$ denote the central extension of $\mathcal{M}_{g-1,2} $ obtained by $ \iota \circ \xi$.
Since $\mathcal{M}_g$ and $\M_{g-1,2} $ are perfect, the group $(\iota \circ \xi)^* (\mathcal{T}_g)$ is also perfect by the isomorphism \eqref{kkdkl}.
Note that the group $\mathrm{Stab} (\alpha) \cap \Ker (\varepsilon )$ is isomorphic to $ \iota^* (\mathcal{T}_g) $.
Hence the abelianization $ \bigl( \mathrm{Stab} (\alpha) \cap \Ker (\varepsilon ) \bigr) _{\mathrm{ab}}$ never be bigger than $\Z /2 $.
In conclusion, we arrive at the conclusion.
\end{proof}

\subsection{Coxeter quandles}\label{Ex5}
We will focus on Coxeter quandles, and study the associated groups,
and show Theorem \ref{egs}.

This subsection assumes basic knowledge of Coxeter groups, as explained in \cite{Aki,How}.
Given a Coxeter graph $\Gamma $, we can set the Coxeter group $W$.
Let $X_\Gamma $ be the set of the reflections in $W$, that is, the set of elements conjugate to the generators of $W $.
Equipping $X_\Gamma $ with conjugacy operation, $X_\Gamma $ is made into a quandle of type 2.
Denote the inclusion $X_\Gamma
\hookrightarrow W $ by $\kappa$. Since $ W$ subject to the center $Z_W$ effectivity acts on $X_\Gamma $,
we have $\mathrm{Inn}(X_\Gamma ) \cong W/Z_W $.
Moreover, $W $ is, by definition, isomorphic to the quotient of $\mathrm{As}(X_\Gamma )$ subject to the squared relations $(e_x)^2=1$ for any $x \in X_\Gamma $.

In this situation, we now give another easy proof of a part of the theorem shown by Howlett:
\begin{thm}[A connected result in {\cite[\S 2--4]{How}}]\label{egs}
Assume that the Coxeter quandle $X_\Gamma$ is connected.
Then, the second group homology $H_2^{\rm gr}(W) $ is annihilated by 2.
\end{thm}\begin{proof}
Recall from Theorem \ref{cccor1} that $H_1^{\rm gr}(\mathrm{As}(X_\Gamma ))\cong \Z $ and $H_2^{\rm gr}(\mathrm{As}(X_\Gamma ))$ is annihilated by 2.
Therefore, the inflation-restriction exact sequence from the central extension $ \mathrm{As}(X_\Gamma ) \rightarrow W$ implies the desired 2-vanishing of $H_2^{\rm gr}(W) $.
\end{proof}
Finally, we will end this subsection by giving some comments.
Recently, Akita \cite{Aki2} determined the associated group $\mathrm{As} (X_\Gamma )$ as a $\Z^{N}$-central extended group of $ W$.
Furthermore, concerning the third homology $ H_3(\mathrm{As} (X_\Gamma ))$ in the case where $ X_\Gamma$ is connected,
we obtain $H_3(\mathrm{As} (X_\Gamma )) \cong H_3(W)$ up to 2-torsion from Lemma \ref{dfgk232}.
The odd torsion of $H_3^{\rm gr}(W)$ in a certain stable range is studied by Akita \cite{Aki}.

\subsection{Core quandles}\label{Ex6}

Given a group $G$, we let $X=G$ equipped with a quandle operation $g \lhd h:=hg^{-1}h$, which is called {\it core quandle} \cite{Joy} and is of type 2.
This last subsection will deal with core quandles, and show Proposition \ref{eahge33s}.

Let us give some terminologies to state the proposition.
Let $\Z/2$ be $\{\pm 1\}$.
Take the wreath product $ (G \times G) \rtimes \Z/2$, and the commutator subgroup $[G,G]$. 
Consider the epimorphism $(G \times G) \rtimes \Z/2 \rightarrow G/[G,G] $ which sends $(g,h,\sigma)$ to $[gh]$.
Then, the kernel is formed as
$$\mathcal{G}_1 := \{ \ (g,h, \sigma )\ \in (G \times G) \rtimes \Z/2 \ | \ gh \in[G,G] \ \}.$$
Further, with respect to $x\in X$ and $(g,h, \sigma ) \in \mathcal{G}_1 $,
we define $ x \cdot (g,h,\sigma) := h^{-1} x^{\sigma } g$, which ensures
an action of $ \mathcal{G}_1$ on $X$.
Further, consider a subgroup of the form
$$ \mathcal{G}_2 := \bigl\{ (z,z, \sigma ) \in (G\times G) \rtimes \Z/2 \ \bigl| \ \ z^2 \in [G,G], \ \ \ k^{-1} z k= z^{ \sigma }\ \textrm{ for any } k \in G \ \bigr\}, $$
which is contained in the center of $ \mathcal{G}_1 $.
Then, the quotient action subject to $\mathcal{G}_2 $ is effective.
\begin{prop}\label{eahge33s}
There is a group isomorphism $\mathrm{Inn}(X) \cong \mathcal{G}_1 /\mathcal{G}_2 $.
\end{prop}
\begin{proof}
Consider the map $\kappa: X \rightarrow \mathcal{G}_1 /\mathcal{G}_2 $ which sends $g$ to $[(g,g^{-1},-1)]$. We claim that this $\mathcal{G}_1 /\mathcal{G}_2 $ is generated by the image $ \Im (\kappa).$
Actually, we can easily verify that
any element $(g,h, \sigma )$ in $ \mathcal{G}_1 $ with $g_i , \ h_i \in G$ and $g h= g_1 h_1g_1^{-1} h_1^{-1} \cdots g_m h_mg_m^{-1} h_m^{-1} $
is decomposed as
$$ \kappa(1_G)^{\frac{\sigma +1}{2}} \cdot \kappa( g h^{-1}) \cdot
\Bigl( \bigl(\kappa(g_1 h_1)\cdot \kappa(1_G)\cdot \kappa(g_1^{-1})\cdot \kappa(h_1) \bigr) \cdots \bigl(\kappa(g_m h_m)\cdot \kappa(1_G)\cdot \kappa(g_m^{-1})\cdot \kappa(h_m) \bigr) \Bigr).
$$
Then, the routine discussion from Lemma \ref{keykantan1g} completes the proof.
\end{proof}
This proposition implies the difficulty to determine $\mathrm{Inn}(X)$, in general.
Thus, it also seems hard to determine $\mathrm{As} (X)$. Actually, even if $X$ is a connected core quandle,
Proposition \ref{eahge33s} implies that the kernel $\Ker (\psi)$ is complicated by the reason of the second homology $H_2^{\rm gr}(G)$ and $ H_2^{\rm gr}(\mathrm{Inn}(X))$.
For example, if $ X$ is
the product of $h$-copies of the cyclic group $\Z/m $, i.e., $X$ is the Alexander quandle of the form $(\Z/m)^h [T]/(T+1)$,
then the kernel $\Ker (\psi)$ stated in Proposition \ref{H2dfg31sha} is not so simple.

\section{On quandle coverings}\label{SScon32}

This section suggests that the results in section 2 are applicable to quandle coverings.

Let us review coverings in the sense of Eisermann \cite{Eis2,Eis3}.
A map $f: Y \rightarrow Z$ between quandles is {\it a (quandle) homomorphism}, if $f(a\lhd b)=f(a)\lhd f(b)$ for any $a,b \in Y$.
Furthermore, a quandle epimorphism $p: Y \rightarrow Z$ is a {\it (quandle) covering},
if the equality $ p( \widetilde{x})=p(\widetilde{y}) \in Z $ implies $ \widetilde{a} \lhd \widetilde{x} = \widetilde{a} \lhd \widetilde{y} \in Y $ for any $ \widetilde{a}, \ \widetilde{x}, \ \widetilde{y} \in Y$.

Let us mention a typical example.
Given a connected quandle $X$ with $a \in X$,
recall the abelianization $\varepsilon_0 :\mathrm{As} (X) \rightarrow \Z$ in \eqref{epsilon}.
Then, the kernel $\Ker (\varepsilon_0) $ has 
a quandle operation defined by setting
$$ g \lhd h := e_a^{-1} g h^{-1} e_a h \ \ \ \ \ \ \ \mathrm{for} \ g,h \in \Ker(\varepsilon_0) . $$
We can easily see the independence of the choice of $a \in X$ up to quandle isomorphisms.
Ones write $\X$ for the quandle $(\Ker (\varepsilon_0), \lhd)$, which is
considered in \cite[\S 7]{Joy}. When $X$ is of type $t_X $, so is the extended one $\X$ by Lemma \ref{daiji}. 
Furthermore, using the restricted action $ X \curvearrowleft \Ker (\varepsilon_0) \subset \mathrm{As} (X)$,
we see that the map $p: \X \rightarrow X$ sending $g$ to $a \cdot g$ is a covering.
This $p$ is called {\it the universal (quandle) covering of }$X$, according to \cite[\S 5]{Eis2}. 

As a preliminary, we will explore some properties of quandle coverings.
\begin{prop}\label{lemwl}
For any quandle covering $p: Y \rightarrow Z$, the induced group surjection $p_* : \mathrm{As} (Y ) \rightarrow \mathrm{As} (Z)$ is a central extension.
Furthermore, if $Y$ and $Z$ are connected and $Z$ is of type $t_Z $, then the abelian kernel $\Ker(p_*)$ is annihilated by $t_Z $.

\end{prop}
\begin{proof}
Fix a section $\mathfrak{s}: Y \rightarrow Z$.
For any $y \in Z$, put arbitrary $y_i \in p^{-1}(y)$.
Then,
$$ e_{\mathfrak{s}(y) }^{-1} e_{ b} e_{\mathfrak{s}(y) } = e_{ b \lhd \mathfrak{s}(y)} = e_{ b \lhd y_i } = e_{y_i }^{-1}e_{ b} e_{y_i } \in \mathrm{As} (Y) $$
for any $b \in Y$.
Here the second equality is due to the covering $p$. Denoting $e_{\mathfrak{s}(y) } e_{y_{i}}^{-1} $ by $z_{i}$,
the equalities imply that $z_{i}$ is central in $\mathrm{As}(Y)$. Since $e_{\mathfrak{s}(y) } = z_{i}e_{y_{i}} $, $\mathrm{As} (Y)$ is generated by $ e_{\mathfrak{s}(y) }$ with $y \in Y$ and the central elements $z_{i}$ associated with $y_i \in p^{-1}(y)$;
consequently, the surjection $p_*$ is a central extension.

We will show the latter part. Take the inflation-restriction exact sequence, i.e.,
$$ H_2^{\mathrm{gr}}( \mathrm{As} (Z ) ) \stackrel{\ }{\lra} \Ker(p_*) \lra H_1^{\mathrm{gr}}( \mathrm{As} (Y ) ) \stackrel{ }{\lra} H_1^{\mathrm{gr}}( \mathrm{As} (Z ) ) \lra 0 \ \ \ \ (\mathrm{exact}).$$
By connectivities the third map from $H_1^{\mathrm{gr} }(\mathrm{As}(Y)) =\Z$ is an isomorphism.
Since Theorem \ref{cccor1} says that $H_2^{\mathrm{gr}}( \mathrm{As} (Z ) ) $ is annihilated by $t_Z $, so is the kernel $\Ker(p_*)$ as desired. 
\end{proof}

Next, we will compute the second homology of $\X$ (Theorem \ref{proof11ti82}) by showing propositions:
\begin{prop}\label{lem11ti}For any connected quandle $X $,
the extended one $\X$ above is also connected.
\end{prop}
\begin{proof}
It is enough to show that the identity $1_{\X} \in \X= \Ker (\varepsilon_0)$ is transitive to any element $h$ in $\X$.
Expand $h \in \X \subset \mathrm{As}(X)$ as $h =e_{x_1}^{ \epsilon_1} \cdots e_{x_n}^{\epsilon_n}$ for some $x_i\in X$ and $ \epsilon_i \in \Z$.
Since $h \in \Ker (\varepsilon_0) $, note $\sum \epsilon_i =0 $.
The connectivity of $X$ ensures some $g_i \in \mathrm{As}(X) $ so that $a \cdot g_i^{\epsilon_i} = x_i$.
Therefore $ g_i^{- \epsilon_i} e_a g_i^{\epsilon_i} = e_{a \cdot g_i^{\epsilon_i} } =e_{x_i}^{\epsilon_i}$ by (\ref{hasz2}). In the sequel, we have
$$\bigl( \cdots ( 1_{\X} \lhd^{\epsilon_1} g_1) \cdots \lhd^{\epsilon_n} g_n\bigr) = e_a^{\sum \epsilon_i} 1_{\X} ( g_1^{- \epsilon_1} e_a g_1^{\epsilon_1} ) \cdots (g_n^{- \epsilon_n} e_a g_n^{\epsilon_n})= e_{x_1}^{\epsilon_1} \cdots e_{x_n}^{\epsilon_n} =h.$$
These equalities in $\X $ imply the transitivity of $\X $
\end{proof}

\begin{prop}\label{proof11ti}
Let $X$ be a connected quandle.
Let $p_* : \mathrm{As} (\X ) \rightarrow \mathrm{As} (X) $ be the epimorphism induced from the covering $p : \X \rightarrow X $.
Then, under the canonical action of $\mathrm{As} (\X)$ on $\X$, the stabilizer $\mathrm{Stab}(1_{\X} )$ of $1_{\X} $ is equal to
$ \Z \times \Ker (p_*)$ in $\mathrm{As} (\X)$. 
Furthermore, the summand $\Z$ is generated by $1_{\X} $. \end{prop}

\begin{proof}
We can easily see that
the stabilizer of $1_{\X} $ via the previous action $ \Ker (\varepsilon_0)= \X \curvearrowleft \mathrm{As}(X)$ is $ \underline{\mathrm{Stab}}(1_{\X} ) = \{ e_a^{n }\}_{n \in \Z} \subset \mathrm{As}(X)$ exactly.
Notice that any central extension of $\Z$ is trivial; therefore, since $p_*$ is a central extension (Proposition \ref{lemwl}),
the restriction $p_*: \mathrm{Stab}(1_{\X}) \rightarrow \underline{\mathrm{Stab}}(1_{\X}) =\Z$ implies the required identity
$ \mathrm{Stab}(1_{\X} ) = \Z \times \Ker (p_*) $.
\end{proof}
\begin{thm}\label{proof11ti82}
The second quandle homology of the extended quandle $\X$ is isomorphic to the kernel of $p_* : \mathrm{As} (\X ) \rightarrow \mathrm{As} (X) $. Namely, $H_2^Q(\X) \cong \Ker (p_*)$.
In particular, it follows from Proposition \ref{lemwl} that, if $t_X < \infty $, then $H_2^Q(\X) $ is annihilated by the type $t_X $.
\end{thm}
\begin{proof} Note that $\X $ is connected (Proposition \ref{lem11ti}) and the kernel $\Ker (p_*)$ is abelian (Proposition \ref{lemwl}).
Accordingly, the desired isomorphism
$H_2^Q(\X) \cong \bigl( \Ker (\varepsilon_{\X}) \cap \mathrm{Stab}(1_{\X}) \bigr)_{\mathrm{ab}} = \Ker (p_*)$
follows immediately from Proposition \ref{proof11ti} and
Theorem \ref{ehyouj2i3}. \end{proof}
Finally, we now discuss the third group homology.
\begin{prop}\label{kunosan}The universal covering $p: \X \rightarrow X$
induces a $[1/t_X]$-isomorphism $p_*: H_3^{\mathrm{gr}} (\mathrm{As}(\X)) \cong H_3^{\mathrm{gr}} (\mathrm{As}(X)) $.\end{prop}
\begin{proof}
By connectivity of $\X$ and Theorem \ref{cccor1}, $ H_2^{\mathrm{gr}} (\mathrm{As}(\X)) $ and $H_2^{\mathrm{gr}} (\mathrm{As}(X)) $ are annihilated by $t_X$.
Furthermore, since the epimorphism $p_*: \mathrm{As}(\X) \rightarrow \mathrm{As}(X)$ is a central extension whose kernel is annihilated by $t_X $ (Proposition \ref{lemwl}), we readily obtain the $[1/t_X]$-isomorphism $p_*: H_3^{\mathrm{gr}} (\mathrm{As}(\X)) \cong H_3^{\mathrm{gr}} (\mathrm{As}(X)) $ from the Lyndon-Hochschild sequence of $p_* $.
\end{proof}
These properties played a key role to prove the main theorem in \cite{Nos1}.

\section{Proof of Theorem \ref{cccor1}.}\label{kokoroni}

The purpose of this section is to prove Theorem \ref{cccor1}.
Let us begin by reviewing the rack space introduced by Fenn-Rourke-Sanderson \cite{FRS1}.
Let $X$ be a quandle with discrete topology.
We set up a disjoint union $ \bigcup_{n \geq 0} ([0,1]\times X)^n $, and consider
the relations given by
\begin{equation}\label{equ.relation} \! (t_1,x_1, \dots, x_{j-1},0,x_j,t_{j+1},\dots, t_n,x_n) \sim (t_1,x_1, \dots t_{j-1},x_{j-1},t_{j+1},x_{j+1},\dots, t_n, x_{n} ). \notag \end{equation}
$ \ \ (t_1,x_1, \dots, x_{j-1},1,x_j, t_{j+1}, \dots, t_n,x_n) \sim$
\[ \ \ \ \ \ \ \ \ \ \  ( t_1,x_1\lhd x_{j}, \dots, t_{j-1}, x_{j-1}\lhd x_{j},t_{j+1},x_{j+1},\dots, t_n,x_n), \]
Then, the {\it rack space} $BX$ is defined to be the quotient space.
By construction, we have a cell decomposition of $BX$ by
regarding the projection $ \bigcup_{n \geq 0} ([0,1]\times X)^n \rightarrow BX$
as characteristic maps.
From the 2-skeleton of $BX$, we have $\pi_1(BX) \cong \mathrm{As} (X)$.
Considering the Eilenberg-MacLane space $ K( \pi_1(BX),1) )$, 
we have the classifying map $c : BX \hookrightarrow K (\pi_1(BX),1) $, i.e., an conclusion obtained by killing the higher homotopy groups of $BX$.

\begin{thm}\label{centthm}
Let $X$ be a connected quandle of type $t $, and let $t < \infty $.
For $n
= 2$ and $3$, the induced map $c_*: H_n(BX ) \rightarrow H_n^{\rm gr} (\mathrm{As}(X)) $ is annihilated by $t$.
\end{thm}
\begin{rem}\label{clarem} This is still more powerful and general than a result of Clauwens \cite[Proposition 4.4]{Cla},
which stated that, if a quandle $X$ of finite order satisfies a certain condition, then
the composite $(\psi_X)_* \circ c_*: H_n(BX ) \rightarrow H_n^{\rm gr} (\mathrm{As}(X)) \rightarrow H_n^{\rm gr} (\mathrm{Inn}(X)) $ is
annihilated by $|\mathrm{Inn}(X)|/|X|$ for any $n\in \mathbb{N} $.
Here note from Lemma \ref{lem11} that $t$ is a divisor of the order $|\mathrm{Inn}(X)|/|X|$.
\end{rem}

Since the induced map $c_*: H_2(BX ) \rightarrow H_2^{\rm gr} (\mathrm{As}(X)) $ with $n=2$ is known to be surjective (cf. Hopf's theorem \cite[II.5]{Bro}),
Theorem \ref{cccor1} is immediately obtained from Theorem \ref{centthm} and the inflation-restriction exact sequence of \eqref{AI}.
Hence, we may turn into proving Theorem \ref{centthm}.

To this end, we give a brief review of the rack and quandle homology.
Let $C_n^R(X)$ be the free right $\Z$-module generated by $ X^n$.
Define a boundary $\partial^R_n : C_n^R(X) \rightarrow C_{n-1}^R(X )$ by
\begin{eqnarray*}
\begin{array}{rl}
\partial^R_n ( x_1, \dots,x_n)= \sum_{1\leq i \leq n} (-1)^i\bigl(& ( x_1\lhd x_i,\dots,x_{i-1}\lhd x_i,x_{i+1},\dots,x_n) \\
&-( x_1, \dots,x_{i-1},x_{i+1},\dots,x_n) \bigr).
\end{array}
\end{eqnarray*}
The composite $\partial_{n-1}^R \circ \partial_n^R $ is known to be zero. The homology is denoted by $H^R_n(X)$ and is called {\it the rack homology}.
As is known, the cellular complex of the rack space $BX $ is isomorphic to the complex $(C_* ^R(X),\partial_*^R )$. In particular, we have the isomorphism $H_*(BX) \cong H_*^R(X) $.
Furthermore, following \cite{CJKLS}, let $C^D_n (X) $ be a submodule of $C^R_n (X)$ generated by $n$-tuples $(x_1, \dots,x_n)$
with $x_i = x_{i+1}$ for some $ i \in \{1, \dots, n-1\}$. It can be easily seen that the submodule $C^D_n (X) $ is a subcomplex of $ C^R_{n} (X). $
Then the {\it quandle homology}, $H^Q_n (X) $, is defined to be the homology of the quotient complex $C^R_n (X ) /C^D_n (X)$.

Furthermore,
we now observe concretely the map $c_*: H_n(BX ) \rightarrow H_n^{\rm gr} (\mathrm{As} (X)) $ for $n \leq 3$.
Let us recall the (non-homogenous) standard complex $ C^{\rm gr}_n (\mathrm{As}(X))$ of $\mathrm{As} (X)$; see e.g. \cite[\S I.5]{Bro}.
The map $c_*$ can be described in terms of their complexes.
In fact, Kabaya \cite[\S 8.4]{Kab} considered homomorphisms $c_n : C_n^R(X) \rightarrow C_n^{\rm gr} (\mathrm{As}(X)) $,
where the map $c_n$ for $n\leq 3$ are defined by setting
\begin{eqnarray*}
\begin{array}{rl}
c_1(x)=&e_x,\\
c_2(x,y)=&(e_x,e_y) -(e_y, e_{x\lhd y}),\\
c_3(x,y,z)=&(e_x,e_y,e_z ) -(e_x,e_z, e_{y \lhd z})+(e_y,e_z,e_A ) -(e_y ,e_{x \lhd y }, e_z )+ (e_z ,e_{x \lhd z },e_{y \lhd z} ) \\
&- (e_z,e_{y \lhd z} ,e_A ),
\end{array}
\end{eqnarray*}
where we denote $(x\lhd y ) \lhd z \in X $ by $A $ for short.
As is shown (see \cite[\S 8.4]{Kab}), the induced map on homology coincides with the map above $c_*$ up to homotopy.

We will construct a chain homotopy between $t \cdot c_n $ and zero, when $X$ is connected and of type $t$.
Define a homomorphism $h_i : C_i^R(X ) \rightarrow C_{i+1}^{\rm gr} (\mathrm{As}(X)) $ by setting
\[h_1(x)= \sum_{ 1 \leq j \leq t-1} \!\!(e_x, e_x^j ), \]
\[h_2(x,y)=\sum_{ 1 \leq j \leq t-1} \!\! (e_x, e_y, e_{x\lhd y}^j ) - (e_x, e_x^j, e_{y} ) - (e_y, e_{x \lhd y}, e_{x \lhd y}^j ) + (e_y, e_y^j, e_{y}) , \]
\[ h_3(x,y,z)= \sum_{ 1 \leq j \leq t-1}\!\! \bigl( \!(e_x, e_y, e_z, e_{A }^j ) - (e_x, e_z, e_{y \lhd z }, e_{A }^j ) - (e_x, e_y, e_{x \lhd y}^j , e_{z} ) - (e_y, e_{x \lhd y}, e_z, e_{A}^j ) \]
\[ \ \ \ \ \ \ \ \ \ \ + (e_x,e_z, e_{ x \lhd z }^j , e_{y \lhd z}) + (e_z ,e_{ x \lhd z } , e_{y \lhd z}, e_A^j) + (e_x,e_x^j , e_{ y } , e_{ z}) - (e_x,e_x^j , e_{ z } , e_{y \lhd z}) \]
\[ \ \ \ \ \ \ \ \ \ \ + (e_y,e_z , e_{ A} , e_{A}^j) - (e_z,e_{y \lhd z }, e_{ A} , e_{A}^j) - (e_z,e_{x \lhd z} , e_{ x \lhd z }^j , e_{y \lhd z}) + (e_y,e_{x \lhd y} , e_{ x \lhd y }^j , e_{z})\bigr) . \]

\begin{lem}\label{tthm21}
Let $X$ be as above. Then we have the equality $ h_1 \circ \partial_2^R - \partial_3^{\rm gr} \circ h_2 = t \cdot c_2 $.
\end{lem}
\begin{proof} Compute the both terms $ h_1 \circ \partial_2^R$ and $ \partial_3^{\rm gr} \circ h_2 $
in the left hand side as
\[ h_1 \circ \partial_2^R (x,y) = \sum (e_x, e_x^j) - (e_{ x \lhd y } , e_{x \lhd y}^j) .\]
\[ \partial_3^{\rm gr} \circ h_2 (x,y)= \partial_3^{\rm gr} \sum \bigl((e_x, e_y, e_{x \lhd y}^j) - (e_x, e_x^j, e_y) - (e_y, e_{ x \lhd y } , e_{x \lhd y}^j) +(e_y,e_y^j,e_y) \bigr) \]
\[ = \bigl( \sum ( e_y, e_{x \lhd y}^j) - (e_x e_y , e_{x \lhd y}^j) + (e_x, e_{ x}^j e_y ) - (e_x ,e_y )- (e_x^j ,e_y )+ (e_x^{j+1} ,e_y )- (e_x ,e_x ^j e_y ) \]
\[ \ \ \ \ \ + (e_x ,e_x^j ) -(e_{ x \lhd y } ,e_{ x \lhd y }^j) + (e_x e_y ,e_{ x \lhd y }^j) - (e_y ,e_{ x \lhd y }^{j+1}) + (e_y ,e_{ x \lhd y }) \bigr)+ ( e_y ,e_y^{t} )-(e_y^{t},e_y) \]
\[ =t \bigl( (e_y,e_{x \lhd y}) -(e_x,e_y)\bigr) + ( e_x^{t}, e_y )-( e_y, e_{x \lhd y}^{t} ) -( e_y^{t} , e_{ y})+ ( e_y, e_{ y}^{t}) + h_1 \circ \partial_2^R (x,y) \]
\[ =- t \cdot c_2(x,y) + h_1 \circ \partial_2^R (x,y) .\]
Here we use Lemma \ref{daiji} for the last equality. Hence, the equalities complete the proof.
\end{proof}

\begin{lem}\label{tthm23}Let $X$ be as above.
The difference $ h_2 \circ \partial_3^R - \partial_4^{\rm gr} \circ h_3 $ is chain homotopic to $ t \cdot c_3 $.
\end{lem}
\begin{proof} This is similarly proved by a direct calculation.
To this end, recalling the notation $A=(x \lhd y ) \lhd z$, we remark two identities
$$ e_z e_A = e_{x \lhd y} e_z , \ \ \ e_{y \lhd z}e_A = e_{x \lhd z} e_{y \lhd z} \in \mathrm{As} (X). $$
Using them, a tedious calculation can show that the difference $ (t \cdot c_3 - h_2 \circ \partial_3^R -\partial_4^{\rm gr} \circ h_3) (x,y,z)$ is equal to
\[ (e_y,e_z,e_A^{t}) - (e_x^{t},e_y,e_z)+( e_x^{t},e_z,e_{e \lhd z}) -(e_y,e_{x \lhd y}^{t})\]
\[ \ + (e_z ,e_{x \lhd z}^{t},e_{y \lhd z}) - (e_z ,e_{y \lhd z} ,e_A^{t} )+ \sum_{1\leq j \leq t-1} (e_y,e_y^j,e_y ) -(e_{y \lhd z},e_{y \lhd z}^j,e_{y \lhd z} ). \]
Note that this formula 
is independent of any $x \in X $ since the identity $(e_a)^{t} =(e_b)^{t} $ holds for any $a,b \in X$ by Lemma \ref{daiji}.
However, the map $c_3 (x,y,z)$ with $x=y$ is zero by definition.
Hence, the map $ t \cdot c_3 $ is null-homotopic as desired.
\end{proof}
\begin{proof}[Proof of Theorem \ref{centthm}] The map $t \cdot c_*$ are obviously null-homotopic by Lemmas \ref{tthm21} and \ref{tthm23}.
\end{proof}

\noindent
The proof was an ad hoc computation in an algebraic way; however the theorem should be easily shown by a topological method:

\vskip 0.7pc

\noindent
{\bf Problem} Does the $t $-vanishing of the map $ c_* : H_n(BX ) \rightarrow H_n^{\rm gr} (\mathrm{As}(X)) $ hold
for any $n \in \mathbb{N}$? Provide its topological proof. Further, how about the non-connected quandles?

\subsection*{Acknowledgment} The author 
is grateful to Toshiyuki Akita for useful discussions on Coxeter groups, and comments
errors in a previous version of this paper.
He also thanks the referee for careful reading the manuscript and for giving useful comments.

\end{document}